\documentclass[a4paper]{article}

\usepackage[colorlinks=false]{hyperref}

\usepackage[english]{babel}
\usepackage[utf8]{inputenc}
\usepackage[T1]{fontenc}
\usepackage{moresize}

\usepackage{amsmath}
\usepackage{amsfonts}
\usepackage{amssymb}
\usepackage{amsthm}
\usepackage{bbm}
\usepackage{biblatex}
\usepackage{csquotes}
\usepackage[bottom]{footmisc}
\usepackage{graphicx}
\usepackage{mathrsfs}
\usepackage{mwe}
\usepackage{pdfpages}
\usepackage{wrapfig}

\allowdisplaybreaks

\usepackage[a4paper,top=3cm,bottom=2cm,left=3cm,right=3cm,marginparwidth=1.75cm]{geometry}

\theoremstyle{definition}
\newtheorem{definition}{Definition}
\newtheorem{theorem}{Theorem}
\newtheorem{remark}{Remark}
\newtheorem{example}{Example}
\newtheorem{question}{Question}

\makeatletter
\renewenvironment{abstract}{%
    \if@twocolumn
      \section*{\abstractname}%
    \else 
      \begin{center}%
        {\bfseries \Large\abstractname\vspace{\z@}}
      \end{center}%
      \quotation
    \fi}
    {\if@twocolumn\else\endquotation\fi}
\makeatother

\title{\textbf{Reconstructing Rooted Trees From Their Strict Order Quasisymmetric Functions}}
\author{\textbf{Jeremy Zhou}\thanks{jzhou21@andover.edu}}
\date{}

\begin{document}

\maketitle

\begin{abstract}
\noindent
Determining whether two graphs are isomorphic is an important and difficult problem in graph theory. One way to make progress towards this problem is by finding and studying graph invariants that distinguish large classes of graphs. Stanley conjectured that his chromatic symmetric function distinguishes all trees, which has remained unresolved. Recently, Hasebe and Tsujie introduced an analogue of Stanley's function for posets, called the strict order quasisymmetric function, and proved that it distinguishes all rooted trees. In this paper, we devise a procedure to explicitly reconstruct a rooted tree from its strict order quasisymmetric function by sampling a finite number of terms. The procedure not only provides a combinatorial proof of the result of Hasebe and Tsujie, but also tracks down the representative terms of each rooted tree that distinguish it from other rooted trees.
\end{abstract}

\textbf{Keywords:} chromatic symmetric function, $(P, \omega)$-partition, sampling function, algorithmic reconstruction

\pagebreak

\tableofcontents

\pagebreak

\section{Introduction} \label{introduction}

Determining whether two graphs are isomorphic is a very important and difficult problem in graph theory \cite{kobler2012graph}. For instance, in the field of computer vision, graphs can be used to encode visual information, and knowing whether two graphs are isomorphic is crucial for recognizing visual patterns \cite{tonioni2017product}.

To better understand when two graphs could be isomorphic, graph invariants are a useful tool. A \textbf{graph invariant} is a function on graphs that maps any two isomorphic graphs to the same image. Graph invariants can take values in any set, but in this paper all graph invariants will take values that are polynomials or formal power series. For a set of graphs $S$, a graph invariant \textbf{distinguishes elements of $S$} if any two graphs in $S$ mapping to the same image are isomorphic. The existence of such a graph invariant would reduce the graph isomorphism problem for elements of $S$ to calculating the value of the invariant.

One of the most well-known graph invariants is the \textbf{chromatic polynomial} $\chi_G(x)$. It was defined by Birkhoff as the unique polynomial such that $\chi_G(n)$ is the number of ways to properly color $G$ with $n$ colors \cite{birkhoff1912determinant}. The chromatic polynomial is not powerful enough to distinguish every graph, however, because there are many examples of pairs of graphs with the same chromatic polynomial. In particular, all trees with a fixed number of vertices have the same chromatic polynomial: a tree $T$ with $d$ vertices has chromatic polynomial $\chi_T(x) = x(x - 1)^{d - 1}$.

Stanley defined a generalization of the chromatic polynomial, which he named the \textbf{chromatic symmetric function} $X_G(\mathbf{x})$ for $\mathbf{x}$ an infinite tuple of variables $(x_1, x_2, \dots)$ \cite{stanley1995symmetric}:
$$X_G(\mathbf{x}) = \sum_{f \colon V(G) \to \mathbb{Z}^+} \mathbf{x}_f,$$
where $\mathbf{x}_f = \prod_{v \in V(G)} x_{f(v)}$.

Because the chromatic symmetric function has infinitely many variables, it is no surprise that the chromatic symmetric function is in general better than the chromatic polynomial at telling apart graphs. However, the chromatic symmetric function does not distinguish all graphs; as Stanley notes, the following two graphs have the same chromatic symmetric function \cite{stanley1995symmetric}.

\begin{figure}[h]
    \centering
    \begin{minipage}{0.45\textwidth}
        \centering
        \includegraphics{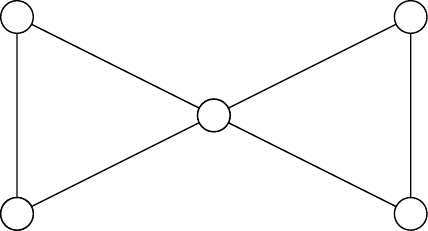}
        \caption{The bowtie graph.}
    \end{minipage}\hfill
    \begin{minipage}{0.45\textwidth}
        \centering
        \includegraphics{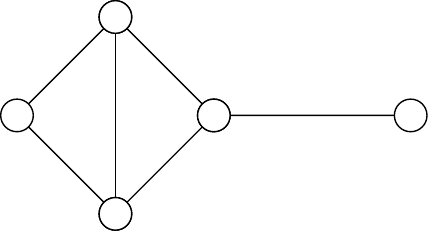}
        \caption{The dart graph.}
    \end{minipage}
\end{figure}

Stanley posed the following question \cite{stanley1995symmetric}:

\begin{question} \label{stanley}
Does the chromatic symmetric function distinguish all trees?
\end{question}

In the years since, significant strides have been made towards a solution.

One approach is to connect the chromatic symmetric function with other invariants. Martin, Morin, and Wagner showed that the chromatic symmetric function is a stronger invariant than the subtree polynomial \cite{martin2008distinguishing}, which was shown by Eisenstat and Gordon to distinguish spiders and some caterpillars \cite{eisenstat2006non}. Aliste-Prieto and Zamora connected the chromatic symmetric function on proper caterpillars to the $\mathscr{L}$-polynomial on integer compositions, allowing them to show that the function distinguishes proper caterpillars \cite{aliste2014proper}.

Another approach is to create a recurrence for the chromatic symmetric function, allowing results to be proven recursively. Gebhard and Sagan generalized the chromatic symmetric function to noncommutative variables, allowing them to use a deletion-contraction relation to prove generalizations of some results of Stanley \cite{gebhard2001chromatic}. Orellana and Scott demonstrated a three-term recurrence relation for chromatic symmetric functions \cite{orellana2014graphs}.

A third approach is to devise a procedure that can reconstruct enough data from the chromatic symmetric function to distinguish certain classes of trees. Loebl and Sereni devised a procedure showing that the chromatic symmetric function distinguishes all caterpillars \cite{loebl2018isomorphism}.

However, Question \ref{stanley} remains unsolved and is actively being researched. Recent results include the following. Heil and Ji computationally verified Question \ref{stanley} in the affirmative up to $29$ vertices \cite{heil2019algorithm}. Huryn determined that the chromatic symmetric function distinguishes $2$-spiders \cite{huryn2020few}. Crew and Spirkl generalized the chromatic symmetric function and the related Tutte symmetric function to vertex-weighted trees, allowing them to use a deletion-contraction relation to prove various new results about both invariants \cite{crew2020deletion, crew2020vertex}.

The chromatic symmetric function is also studied for its connection to knot theory. Noble and Welsh found that their $W$-polynomial, which was originally developed for its connection with Vassiliev invariants of knots, is equivalent to the chromatic symmetric function for trees \cite{noble1999weighted}.

Furthermore, the chromatic symmetric function is studied for its connections to representation theory. Another commonly studied conjecture regarding the chromatic symmetric function is the Stanley-Stembridge conjecture (the $e$-positivity conjecture) \cite{stanley1995symmetric}, which was originally related to immanants.

Shareshian and Wachs generalized the chromatic symmetric function to their \textbf{chromatic quasisymmetric function} for labeled graphs, through which they find a connection with Hessenberg varieties \cite{shareshian2016chromatic}. This allowed them to approach the Stanley-Stembridge conjecture from the angle of representation theory. Harada and Precup developed this connection further by considering a graded version of the conjecture, inspired by the gradation of the cohomology ring of Hessenberg varieties \cite{harada2019cohomology}. In addition, Ellzey generalized the chromatic quasisymmetric function to directed graphs \cite{ellzey2017directed}.

Now, we introduce a particular invariant that we study in this paper. Hasebe and Tsujie defined an analogue of Stanley's chromatic symmetric function for posets, which they call the \textbf{strict order quasisymmetric function} $\Gamma^<(P; \mathbf{x})$ \cite{hasebe2017order}. It is defined as follows:
$$\Gamma^<(P; \mathbf{x}) = \sum_{\substack{f \colon V(P) \to \mathbb{Z}^+ \\ f \textrm{ increasing}}} \mathbf{x}_f.$$

In addition to being a direct analogue of the chromatic symmetric function, the strict order quasisymmetric function is a specialization of the chromatic quasisymmetric function defined by Ellzey \cite{ellzey2017directed}, achieved by taking only the terms with the maximal powers of $t$.

The strict order quasisymmetric function is also a specialization of the \textbf{$(P, \omega)$-partition generating function}. Stanley introduced $(P, \omega)$-partitions for labeled posets $(P, \omega)$ as a way to combine many disparate fields of combinatorics, including as a generalization of graph colorings and skew diagrams \cite{stanley1972ordered}. The $(P, \omega)$-partition generating function was further studied by Gessel \cite{gessel1984multipartite}, as well as McNamara and Ward, who demonstrated necessary conditions and separate sufficient conditions under which two labeled posets the same $(P, \omega)$-partition generating function \cite{mcnamara2014equality}. These functions were also studied by Liu and Weselcouch for their connection to the Hopf algebra of posets \cite{liu2019p}, as well as their expansion in the type 1 quasisymmetric power sum basis \cite{liu2020p}. In the latter paper, Liu and Weselcouch proved that the $(P, \omega)$-partition generating function distinguishes series-parallel posets. The $(P, \omega)$-partition generating function reduces to the strict order quasisymmetric function if $P$ is naturally labeled.

Hasebe and Tsujie proved with algebraic methods that the strict order quasisymmetric function distinguishes all rooted trees, considered as posets \cite{hasebe2017order}. Furthermore, Tsujie used a similar method to prove that the chromatic symmetric function distinguishes trivially perfect graphs \cite{tsujie2018chromatic}.

The strict order quasisymmetric function has an infinite number of terms; for computational applications, we may only be able to sample one term at a time. Because the method of Hasebe and Tsujie relies on unique factorization, it does not provide a way to distinguish rooted trees by sampling a finite number of terms from their strict order quasisymmetric functions. Thus, it is of interest to study what exactly \textit{can} be determined about a rooted tree by sampling a finite number of terms from its strict order quasisymmetric function.

Cai, Slettnes, and the author work on this question by introducing a construction that they term \textbf{introducing gaps} \cite{cai2020combinatorial}. This construction takes two positive integers and produces a coloring of the rooted tree. By sampling the strict order quasisymmetric function for the terms associated with the colorings that result from the construction, they are able to reconstruct partial information about the tree.

In our paper, we set up a new framework for the introducing gaps construction, which allows us to recursively extend the construction in a precise manner and to an arbitrary finite degree. Our extended construction takes any multiset of vertices of the rooted tree and produces a coloring of the rooted tree. This construction is broad enough that we can designate \textit{every} coloring of the rooted tree as the result of the construction for some multiset of vertices. We then show how to sample terms to systematically determine information about certain vertices. Through a careful recursive combinatorial argument, we are able to reconstruct complete information about the rooted tree in a finite number of samples.

\begin{theorem} \label{fin}
Any rooted tree can be reconstructed by sampling a finite number of terms from its strict order quasisymmetric function.
\end{theorem}

Note that the result of Hasebe and Tsujie states: given the strict order quasisymmetric function of a rooted tree, there exists exactly one rooted tree corresponding to it \cite{hasebe2017order}. In contrast, our result explicitly reconstructs the corresponding rooted tree via sampling a finite number of terms from the given strict order quasisymmetric function.

This procedure provides a combinatorial proof that the strict order quasisymmetric function distinguishes rooted trees. The strict order quasisymmetric function has been studied in terms of its expansion in the monomial basis \cite{hasebe2017order}, the fundamental basis \cite{liu2020p}, and the power sum basis \cite{liu2019p}. However, the function has not been studied using the terms themselves.

A benefit of analyzing the strict order quasisymmetric function in this manner is that combinatorial techniques require a lesser depth of knowledge to understand than algebraic techniques. This gives mathematicians who are less experienced with (quasi)symmetric functions, as well as algebraic combinatorics in general, the ability to contribute to current and relevant research.

In addition, the finite collection of terms that are sampled during this procedure can serve as a finite representative collection of terms for each rooted tree, which distinguish it from other rooted trees. Because these representative collections are finite, they can be directly compared to distinguish two rooted trees in a way that is computationally feasible.

In Section \ref{background} of this paper, we go over definitions and notations. Then, in Section \ref{example}, we provide an example of our procedure in action. In Section \ref{framework}, we set up the framework for our procedure, and in Sections \ref{main} through \ref{main2}, we prove our main result. Finally, in Section \ref{future}, we state some future directions for this project.

\section{Background and notation} \label{background}

We begin by going over definitions and notations. Some are taken from \cite{hasebe2017order} and \cite{cai2020combinatorial}, though importantly, we change the profile notation from the latter to make it easier to work with.

We notate multisets such that $\{v^e\}$ represents that the element $v$ appears $e$ times in the multiset.

\subsection{Tree-statistics}

For a rooted tree $T$, we use the symbol $v_T$ to denote its root. For a vertex $v \in V(T)$, we denote the subtree induced by $v$ as $S_v$.

\begin{figure}[h]
\centering
\includegraphics{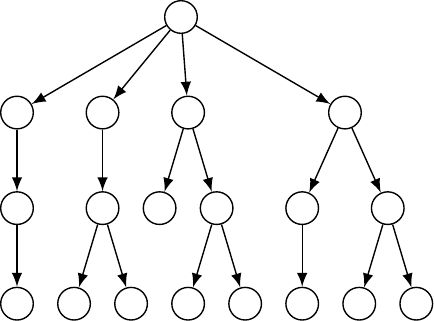}
\caption{A rooted tree.}
\label{fig:tree}
\end{figure}

\begin{definition}
A \textbf{tree-statistic} is a function $a: V(T) \to A$ for some set $A$. We write $a_v$ for the image of $v$ under $a$.
\end{definition}

In this paper, the main tree-statistic that we consider is the coheight.

\begin{definition}
The \textbf{coheight} is a function $h: V(T) \to \mathbb{N}$ defined such that $h_v$ is the length of the unique path from $v_T$ to $v$.
\end{definition}

\begin{definition}
For $n \in \mathbb{N}$, \textbf{layer $n + 1$ of $T$} is the set of vertices of $T$ with coheight $n$. We say that $T$ has \textbf{$N$ layers} if it has $N$ nonempty layers.
\end{definition}

\subsection{Profiles}

For a set of indeterminates $\{ x_i \}_{i \in A}$ indexed by a set $A$, we denote by $\langle x_i \rangle_{i \in A}$ the multiplicative group generated by $\{ x_i \}_{i \in A}$.

\begin{definition}
Let $a \colon V(T) \to A$ be a tree-statistic, and let $\{ x_i \}_{i \in A}$ be a set of indeterminates indexed by $A$. The \textbf{$a$ profile}, denoted $\mathbf{x}_a$, is

$$\mathbf{x}_a = \prod_{v \in V(T)} x_{a_v}.$$

For $v \in V(T)$, the \textbf{$a$ profile of $v$}, denoted $\mathbf{x}_a |_v$, is the $a$ profile of $S_v$:

$$\mathbf{x}_a |_v = \prod_{u \in V(S_v)} x_{a_u}.$$
\end{definition}

For example, we could talk about the \textbf{coheight profile}, denoted $\mathbf{x}_h$, or the \textbf{coheight profile of a vertex $v$}, denoted $\mathbf{x}_h |_v$.

Profiles can also be considered tree-statistics: given a tree-statistic $a: V(T) \to A$, then we can let $\mathbf{x}_a: V(T) \to \langle x_i \rangle_{i \in A}$ be the tree-statistic such that the image of $v$ under $\mathbf{x}_a$ is $\mathbf{x}_a |_v$. Thus, we can nest profiles. For instance, we could consider the \textbf{coheight profile profile} $\mathbf{x}_{\mathbf{x}_h}$.

\begin{example}
For the tree depicted in Figure \ref{fig:tree}, the coheight profile $\mathbf{x}_h$ is ${x_0}^1 {x_1}^4 {x_2}^6 {x_3}^8$, and the coheight profile profile $\mathbf{x}_{\mathbf{x}_h}$ is
$$x_{{x_0}^1 {x_1}^4 {x_2}^6 {x_3}^8} \cdot x_{{x_1}^1 {x_2}^1 {x_3}^1} \cdot x_{{x_1}^1 {x_2}^1 {x_3}^2} \cdot x_{{x_1}^1 {x_2}^2 {x_3}^2} \cdot x_{{x_1}^1 {x_2}^2 {x_3}^3} \cdot x_{{x_2}^1} \cdot x_{{x_2}^1 {x_3}^1} \cdot x_{{x_2}^1 {x_3}^1} \cdots$$
\end{example}

We will eventually want to nest $\mathbf{x}_{\cdots \mathbf{x}_h}$ an arbitrarily large number of times, so we introduce the following notation. For $N \in \mathbb{Z}^+$, let $\mathbf{x}_{(N)h} = \mathbf{x}_{\cdots \mathbf{x}_h}$ nested $N$ times. For example, $\mathbf{x}_{(2)h} = \mathbf{x}_{\mathbf{x}_h}$.

Note that in \cite{cai2020combinatorial}, profiles were defined as multisets and denoted $P^a_T$ and $P^a_v$, which we have replaced with $\mathbf{x}_a$ and $\mathbf{x}_a |_v$, respectively. The definitions contain the same information, but our definition of profile allows us to combine profiles neatly into formal power series. In fact, we will demonstrate that the strict order quasisymmetric function is such a formal power series.

\subsection{Working with formal power series}

Denote by $\mathbb{Z}[[x_i]]_{i \in A}$ the set of formal power series in $\{x_i\}_{i \in A}$ with coefficients in $\mathbb{Z}$.

Let $A$ be a well-ordered set, and let $(x_i)_{i \in A}$ be a sequence of indeterminates indexed by $A$. We can impose a well-order on the set $\langle x_i \rangle_{i \in A}$ by considering each term $\prod_{i \in A} x_i^{e_i}$ ($e_i \in \mathbb{N}$) as the tuple $(e_i)_{i \in A}$ and ordering them lexicographically.

For some formal power series $p \in \mathbb{Z}[[x_i]]_{i \in A}$, we let $\max^m(p)$ be the $m$th greatest term of $p$ under the above ordering. For instance, $\max^2(2x_1 + x_2) = x_1$. Similarly, we let $\min^m(p)$ be the $m$th least term of $p$.

In this paper, we often use formal power series that collect together a set of coheight profiles. For instance, let us fix an $n \in \mathbb{N}$ and construct the formal power series
$$\sum_{\substack{v \in V(T) \\ h_v = n}} \mathbf{x}_h |_v.$$
Since coheight profiles are monomials, we can talk about the term $\min^1(p)$, which is the least coheight profile out of all the coheight profiles of the vertices with coheight $n$.

For $p \in \mathbb{Z}[[x_i]]_{i \in A}$ and $n \in A$, we let
$$\left[ \prod_{i \le n} x_i^{e'_i} \right] p$$
be the formal power series consisting of the terms $\prod_{i \in A} x_i^{e_i}$ in $p$ such that $e_i = e'_i$ for all $i \le n$. For example, if
$$p = x_2 + 2 x_1 x_2 + 3 x_1^2 x_2,$$
then $[x_1] p = 2 x_1 x_2$ and $[x_2] p = x_2$ (since we require that the exponent of $x_1$ is $0$).

\subsection{The strict order quasisymmetric function}

\begin{definition}
A \textbf{coloring} of a rooted tree $T$ is a function $f \colon V(T) \to \mathbb{Z}^+$.

The coloring $f$ is \textbf{increasing} if $f(u) < f(v)$ for every vertex pair $(u, v)$ such that $u$ is a parent of $v$.
\end{definition}

Notice that a coloring $f$ can also be considered a tree-statistic with a slight abuse of notation: $f_v = f(v)$. Thus, we can consider the $f$ profile $\mathbf{x}_f$. See Figure \ref{fig:ex} for an example.

\begin{definition}
The \textbf{strict order quasisymmetric function} of a rooted tree $T$ is the series
$$\Gamma^<(T; \mathbf{x}) = \sum_{\substack{f \colon V(T) \to \mathbb{Z}^+ \\ f \textrm{ increasing}}} \mathbf{x}_f.$$
\end{definition}

\begin{figure}[h]
\centering
\includegraphics{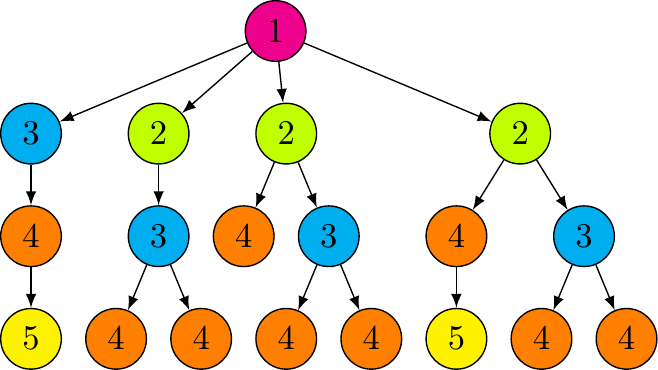}
\caption{A coloring $f$ with $\mathbf{x}_f = {x_1}^1 {x_2}^3 {x_3}^4 {x_4}^9 {x_5}^2$.}
\label{fig:ex}
\end{figure}

\subsection{The sampling function}

In order to work with the strict order quasisymmetric function practically, we need a way of sampling and working with only a finite number of its terms. Thus, we formally introduce the notion of a sampling function.

We denote the set of terms in $\Gamma^<(T; \mathbf{x})$ by $\Gamma^<(T)$.

\begin{definition} \label{sampling}

A \textbf{sampling function of $\Gamma^<(T; \mathbf{x})$} is a function $F \colon S \to \Gamma^<(T) \cup \{ \varnothing \}$, where $S$ is a set.
\end{definition}

This sampling function indexes all the terms in $\Gamma^<(T)$, allowing us to isolate specific terms. With the aid of a sampling function, we can work with a finite number of terms at a time.

In this paper, we use the sampling function $F \colon \langle x_i \rangle_{i \in \mathbb{Z}^+} \to \Gamma^<(T) \cup \{ \varnothing \}$ defined by

$$F \left( \prod_{i \le n} x_i^{e_i} \right) = {\max}^1 \left( \left[ \prod_{i \le n} x_i^{e_i} \right] \Gamma^<(T; \mathbf{x}) \right).$$

Hereafter, we refer to $F$ as \textit{the} sampling function.

We choose this particular sampling function for the purpose of reconstructing a rooted tree using our method. In the following sections, we will elaborate on exactly how it is used.

\section{A reconstruction example} \label{example}

Before proceeding with the formal framework of this paper, we give an example of the reconstruction procedure in action.

\begin{figure}[h]
\centering
\includegraphics{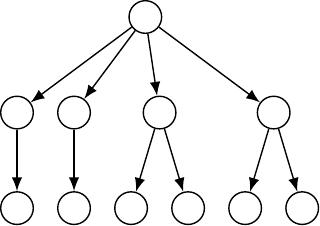}
\caption{A tree $T$.}
\label{fig:extree}
\end{figure}

We begin with the sampling function of $\Gamma^<(T; \mathbf{x})$ for the tree $T$ depicted in Figure \ref{fig:extree}.
Some of the terms of $\Gamma^<(T; \mathbf{x})$ are the following:
$$\Gamma^<(T; \mathbf{x}) = {x_1}^1{x_2}^4{x_3}^6 + 6{x_1}^1{x_2}^4{x_3}^5{x_4}^1 + 6{x_1}^1{x_2}^4{x_3}^5{x_4}^0{x_5}^1 + \cdots + 15{x_1}^1{x_2}^4{x_3}^4{x_4}^2 + \cdots.$$

We will reconstruct the tree $T$ in two steps, accessing a total of five terms.

\subsection{Step 1} \label{step1}

The first step of the reconstruction is to determine the term of $\Gamma^<(T; \mathbf{x})$ with the lexicographically greatest tuple of exponents, which in our notation is $\max^1(\Gamma^<(T; \mathbf{x}))$.

By evaluating the sampling function $F$ at $1$, we determine that $\max^1(\Gamma^<(T; \mathbf{x})) = {x_1}^1{x_2}^4{x_3}^6$. Figure \ref{fig:col1} depicts the coloring of $T$ to which this term corresponds.

\begin{figure}[h]
\centering
\includegraphics{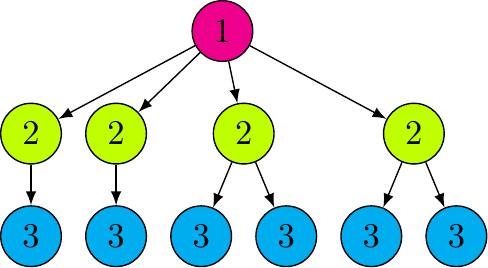}
\caption{The coloring of $T$ corresponding to the term ${x_1}^1{x_2}^4{x_3}^6$.}
\label{fig:col1}
\end{figure}

We will show in Theorem \ref{one} that from the term ${x_1}^1{x_2}^4{x_3}^6$, we know that the coheight profile of $T$ is ${x_0}^1{x_1}^4{x_2}^6$. Thus, we know that the root of $T$ has 4 children and 6 grandchildren. This knowledge is depicted in Figure \ref{fig:stage1}.

\begin{figure}[h]
\centering
\includegraphics{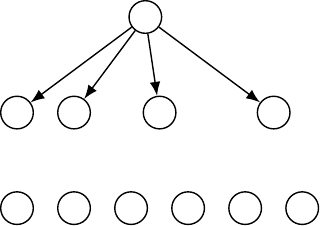}
\caption{The result of step 1 of the reconstruction.}
\label{fig:stage1}
\end{figure}

\subsection{Step 2} \label{step2}

The second step of the reconstruction involves perturbing the first step.
Recall that from the first step, we had $\max^1(\Gamma^<(T; \mathbf{x})) = {x_1}^1{x_2}^4{x_3}^6$. Notice that if we remove all terms except those containing ${x_1}^1{x_2}^4$, the above term would still be the maximum; in our notation, $\max^1([{x_1}^1{x_2}^4] \Gamma^<(T; \mathbf{x})) = {x_1}^1{x_2}^4{x_3}^6$.

However, what happens if we reduce the exponent of $x_2$ to $3$? Let us look at the term $\max^1([{x_1}^1{x_2}^{3}] \Gamma^<(T; \mathbf{x}))$ instead.

By evaluating $F$ at ${x_1}^1{x_2}^3$, we determine that $\max^1([{x_1}^1{x_2}^3] \Gamma^<(T; \mathbf{x})) = {x_1}^1{x_2}^3{x_3}^6{x_4}^1$. Figure \ref{fig:colgap} depicts the two colorings of $T$ to which this term corresponds.

\begin{figure}[h!]
\centering
\includegraphics{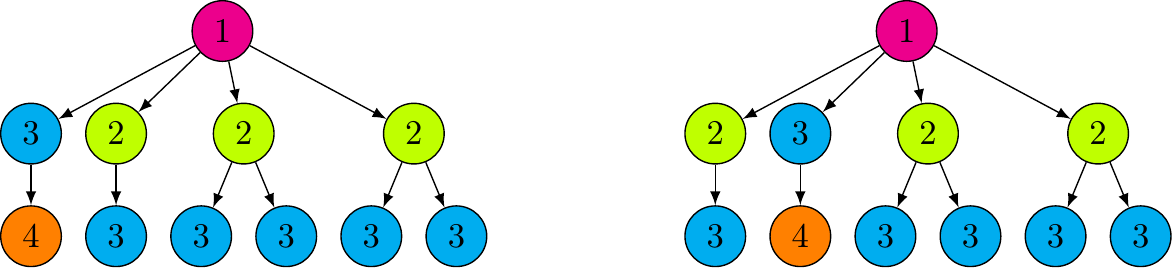}
\caption{The two colorings of $T$ corresponding to the term ${x_1}^1{x_2}^3{x_3}^6{x_4}^1$.}
\label{fig:colgap}
\end{figure}

One can think of this coloring as being similar to the coloring in Figure \ref{fig:col1}, except that the ${x_2}^3$ condition forces there to be a ``gap'' in the number $2$ that is instead filled with a number $3$. Let us call the vertex at which this ``gap'' occurs $g$. We will show in Theorem \ref{two} that by comparing the term ${x_1}^1{x_2}^3{x_3}^6{x_4}^1$ to the term ${x_1}^1{x_2}^4{x_3}^6$ from before, we can determine that the coheight profile of $g$ is ${x_1}^1{x_2}^1$, so $g$ has exactly 1 child.

Information about $g$'s grandchildren, etc.\ can also be deduced in larger cases.
Figure \ref{fig:stage2} summarizes what we now know about $T$.

\begin{figure}[h!]
\centering
\includegraphics{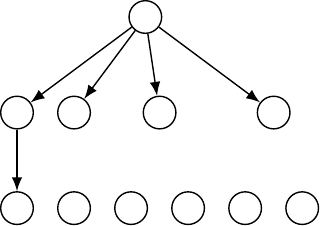}
\caption{The result of the first part of step 2 of the reconstruction.}
\label{fig:stage2}
\end{figure}

We continue the second step of the reconstruction by reducing the exponent of $x_2$ to $2$. Evaluating $F$ at ${x_1}^1{x_2}^2$ gives us the term $\max^1([{x_1}^1{x_2}^2] \Gamma^<(T; \mathbf{x})) = {x_1}^1{x_2}^2{x_3}^6{x_4}^2$. Figure \ref{fig:colgap2} depicts the coloring of $T$ to which this term corresponds.

\begin{figure}[h!]
\centering
\includegraphics{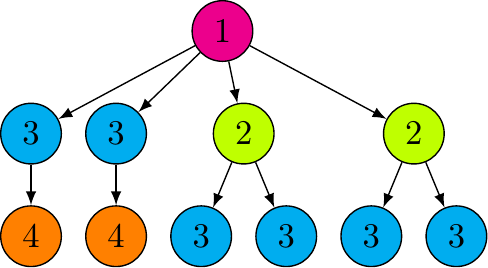}
\caption{The coloring of $T$ corresponding to the term ${x_1}^1{x_2}^2{x_3}^6{x_4}^2$.}
\label{fig:colgap2}
\end{figure}

Let us call the second gap $g'$. By Theorem \ref{two} again, we can compare ${x_1}^1{x_2}^2{x_3}^6{x_4}^2$ to ${x_1}^1{x_2}^4{x_3}^6$ to determine that the product of the coheight profiles of $g$ and $g'$ is ${x_1}^2{x_2}^2$, so $g$ and $g'$ have a total of 2 children. Thus, $g'$ has one child. Figure \ref{fig:stage2.2} summarizes what we now know about $T$.

\begin{figure}[h!]
\centering
\includegraphics{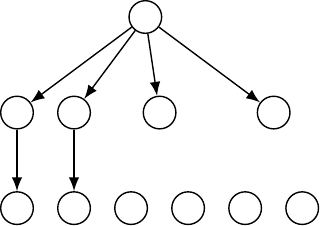}
\caption{The result of the second part of step 2 of the reconstruction.}
\label{fig:stage2.2}
\end{figure}

We can continue this process, determining $\max^1([{x_1}^1{x_2}^1] \Gamma^<(T; \mathbf{x}))$ and $\max^1([{x_1}^1{x_2}^0] \Gamma^<(T; \mathbf{x}))$ in order to reconstruct the number of children of the other vertices in layer $2$. After this, we have the entire tree.

If $T$ is a rooted tree with four or more layers, steps three and up of the reconstruction are mostly analogous. For example, suppose that one of the terms we isolated in step two is ${x_1}^1{x_2}^3{x_3}^6{x_4}^8{x_5}^2$. Considering this, in step three, we might impose the condition $[{x_1}^1{x_2}^3{x_3}^5]$. In total, reconstructing a rooted tree with $n$ layers would require $n - 1$ steps.

\section{A framework for colorings} \label{framework}

What follows is the framework for our main result.

We provide notes explaining how our notation connects to the example in Section \ref{example}.

From a bird's eye view, our proof will take the following steps.
First, we will introduce a special coloring $f_S$ that can be defined for any multiset of vertices $S$. We will then show that some selected terms $\mathbf{x}_{f_S}$ can be isolated from the sampling function. We will demonstrate that from these terms, we can reconstruct $\mathbf{x}_{(N)h}$. Finally, we will show that from $\mathbf{x}_{(N)h}$, we can reconstruct the rooted tree.

In this section, we will accomplish the first proof step and set up a framework for the third.

We begin by defining a coloring $f_\varnothing$, which will act as the base coloring upon which $f_S$ will be constructed. Let $f_\varnothing$ be the coloring such that $f_\varnothing(v) = 1 + h_v$. See Figure \ref{fig:varnothing} for an example.

\begin{figure}[h]
\centering
\includegraphics{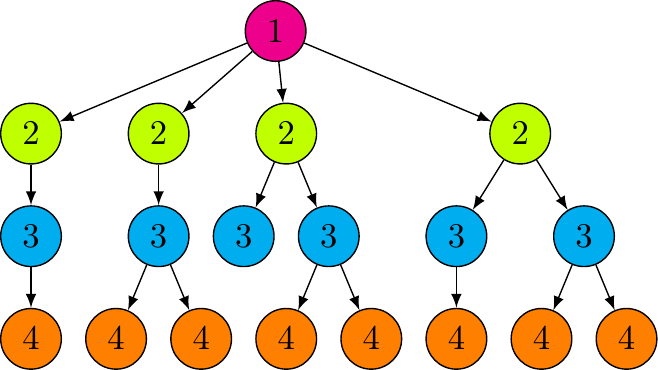}
\caption{Here the coloring $f_\varnothing$ is depicted. Note that $\mathbf{x}_{f_\varnothing} = {x_1}^1 {x_2}^4 {x_3}^6 {x_4}^8$.}
\label{fig:varnothing}
\end{figure}

We need a bit of preliminary notation before defining $f_S$. For two sets $S \subseteq S'$, let the indicator function $\mathbbm{1}_S: S' \to \{0, 1\}$ be the function such that $\mathbbm{1}_S(v) = 1$ if $v \in S$ and $0$ otherwise.

\begin{definition}
For a multiset of vertices $S$, let $f_S$ be the coloring such that $f_S(v) = 1 + h_v + \sum_{g \in S} \mathbbm{1}_{V(S_g)}(v)$.
\end{definition}

Note that $\sum_{g \in S} \mathbbm{1}_{V(S_g)}(v)$ is the number of ancestors of $v$ in $S$.

This construction is a generalization of $f_{n, m}$ from \cite{cai2020combinatorial}, which takes two positive integers $n, m$ and produces a coloring. In contrast, $f_S$ takes any multiset of vertices $S$ and produces a coloring.

We claim that this construction is general enough to encompass every increasing coloring. In other words, for $f$ an increasing coloring of $T$, we can find $S$ such that $f = f_S$. Intuitively, we let $S$ be a measure of how much $f$ ``deviates'' from $f_\varnothing$. In $f_\varnothing$, moving along any edge increases the color by one. If in $f$ moving along an edge increases the color by more than one, then putting a vertex in $S$ accounts for the difference.

\begin{theorem}
\label{f_S}
For every increasing coloring $f \colon V(T) \to \mathbb{Z}^+$, there exists a unique multiset of vertices $S$ such that $f = f_S$. See Figure \ref{fig:gaps} for an example.
\end{theorem}

\begin{figure}[h]
\centering

\includegraphics{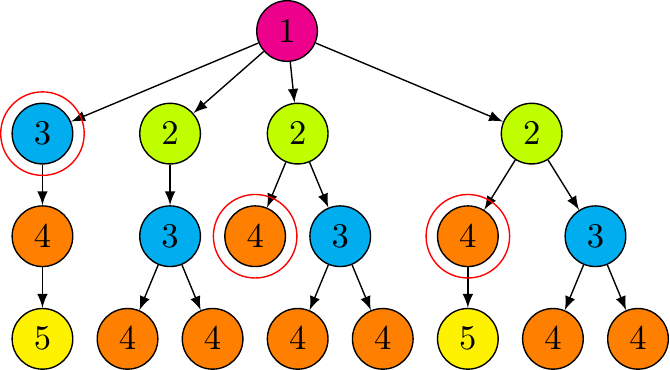}
\caption{The coloring $f$ in Figure \ref{fig:ex} is equal to $f_S$ for $S$ that consists of the circled vertices.}
\label{fig:gaps}
\end{figure}

\begin{proof}
We find the unique $S$ by using recursion on the coheight $h_v$. Begin by considering $h_v = 0$, which includes only the root $v_T$. Since
\begin{align*}
f(v_T) = f_S(v_T) & = 1 + h_{v_T} + \sum_{g \in S} \mathbbm{1}_{V(S_g)}(v_T) \\
& = 1 + \sum_{\substack{g \in S \\ g = v_T}} 1
\end{align*}
we must have that the root $v_T$ appears exactly $f(v_T) - 1$ times in $S$. This works since $f(v_T) - 1 \ge 0$.

Suppose that for $h_v < n$, if the parent of $v$ is $p$, then we have that $v$ must appear $f(v) - f(p) - 1$ times in $S$. This works since $f(v) - f(p) - 1 \ge 0$, which is true because $f$ is increasing. If $v = v_T$, then we set $f(p) = 0$.

Now, let us find for $h_v = n$ how many times $v$ must appear in $S$. Let the path from $v_T$ to $v$ be $p_0 = v_T, p_1, \dots, p_{h_v - 1}, p_{h_v} = v$.

For convenience, let $f(p_{-1}) = 0$. Then
\begin{align*}
f(v) = f_S(v) & = 1 + h_v + \sum_{g \in S} \mathbbm{1}_{V(S_g)}(v) \\
& = 1 + h_v + \sum_{\substack{g \in S \\ v \in S_g}} 1 \\
& = 1 + h_v + \sum_{0 \le i \le h_v - 1} (f(p_i) - f(p_{i - 1}) - 1) + \sum_{\substack{g \in S \\ g = v}} 1 \\
& = 1 + \sum_{0 \le i \le h_v - 1} (f(p_i) - f(p_{i - 1})) + \sum_{\substack{g \in S \\ g = v}} 1 \\
& = 1 + f(p_{h_v}) + \sum_{\substack{g \in S \\ g = v}} 1.
\end{align*}

Thus $v$ appears exactly $f(v) - f(p) - 1$ times in $S$.
\end{proof}

We proceed by setting up a framework for the third proof step, which is to show that from the terms isolated in the second proof step, one can reconstruct $\mathbf{x}_{(N)h}$. The method to do this involves looking at the differences between the terms.

We can do this by expressing $\mathbf{x}_{f_S}$ in terms of $\mathbf{x}_h$ and $\mathbf{x}_h |_g$ for $g \in S$ (Theorem \ref{formula}). This sets up a structure for the rest of our work.

The following theorem establishes the recursive step for Theorem \ref{formula} by expressing $\mathbf{x}_{f_{S'}}$, where $S' = S \cup \{g^1\}$, in terms of $\mathbf{x}_{f_S}$ and $\mathbf{x}_h |_g$.

The idea is that upon adding $g'$, the color of every vertex below $g'$ is shifted up by one.

\begin{theorem} \label{onegap}

Let $S$ be a multiset of vertices. For any $g' \in V(T)$ such that $g'$ has no descendants in $S$, except possibly itself, the following is true. Let $S' = S \cup \{{g'}^1\}$. Let $h' = \sum_{\substack{g \in S}} \mathbbm{1}_{V(S_{g})}(g')$.

Then
$$\mathbf{x}_{f_{S'}} = \mathbf{x}_{f_S} \prod_{v \in V(S_{g'})} \frac{x_{2 + h' + h_v}}{x_{1 + h' + h_v}}.$$

\end{theorem}

\begin{proof}

We know that $f_{S'}(v) = f_S(v) + \mathbbm{1}_{V(S_{g'})}(v)$, so:
$$\mathbf{x}_{f_{S'}} = \mathbf{x}_{f_S} \prod_{v \in V(T)} \frac{x_{f_{S'}(v)}}{x_{f_S(v)}} = \mathbf{x}_{f_S} \prod_{v \in V(S_{g'})} \frac{x_{f_{S}(v) + 1}}{x_{f_S(v)}}.$$
We claim that the numerator and denominator of this expression are equal to the numerator and denominator of the desired expression, respectively. It is sufficient to prove that for every vertex $v \in V(S_{g'})$, it is true that $f_S(v) = 1 + h' + h_v$.

We know that $g'$ has no descendants in $S$. Thus, for all $g \in S$, $S_{g'}$ is either

contained in or has empty intersection with $S_g$. This means that $\mathbbm{1}_{V(S_{g})}(v) = \mathbbm{1}_{V(S_{g})}(g')$ for all $v \in V(S_{g'})$, and
\begin{align*}
f_S(v) & = 1 + h_v + \sum_{\substack{g \in S}} \mathbbm{1}_{V(S_{g})}(v) \\
& = 1 + h_v + \sum_{\substack{g \in S}} \mathbbm{1}_{V(S_{g})}(g') \\
& = 1 + h_v + h'.
\end{align*}
\end{proof}

To clean up the notation, we introduce the \textit{shift function} $\sigma$ and the \textit{shift difference function} $\tau$.

\begin{definition}
The \textbf{shift function} $\sigma: \mathbb{Z}[[x_i]]_{i \in \mathbb{N}} \to \mathbb{Z}[[x_i]]_{i \in \mathbb{N}}$ is defined by

$$\sigma \left( k \prod_{i \in \mathbb{N}} x_i^{e_i} \right) = k \prod_{i \in \mathbb{N}} x_{1 + i}^{e_i}.$$
We denote $s \in \mathbb{N}$ repeated applications of $\sigma$ by $\sigma^s$.
\end{definition}

\begin{definition}
The \textbf{shift difference function} $\tau: \langle x_i \rangle_{i \in \mathbb{N}} \to \langle x_i \rangle_{i \in \mathbb{N}}$

is defined by $\tau(\mathbf{x}) = \frac{\sigma(\mathbf{x})}{\mathbf{x}}$.
\end{definition}

The following are some useful properties of $\sigma$ and $\tau$ that we will use later.

\begin{remark} \label{sigma_multiplicative}
Notice that $\sigma$ is multiplicative; that is, $\sigma(\mathbf{x}_1) \sigma(\mathbf{x}_2) = \sigma(\mathbf{x}_1 \mathbf{x}_2)$. Thus, $\tau$ is also multiplicative.
\end{remark}

\begin{remark} \label{sigma_ordering}
Notice that $\sigma$ preserves the ordering of $\langle x_i \rangle_{i \in \mathbb{N}}$; that is, if $\mathbf{x}_1 < \mathbf{x}_2$, then $\sigma(\mathbf{x}_1) < \sigma(\mathbf{x}_2)$. Notice also that $\tau$ reverses the ordering.
\end{remark}

\begin{remark} \label{sigma_invertible}
Notice that $\sigma$ and $\tau$ are invertible.
\end{remark}

Given these definitions, we can rewrite Theorem \ref{onegap} as follows:
$$\mathbf{x}_{f_{S'}} = \mathbf{x}_{f_S} \prod_{v \in V(S_{g'})} \frac{x_{2 + h' + h_v}}{x_{1 + h' + h_v}} = \mathbf{x}_{f_S} \cdot \sigma( \tau( \sigma^{h'}(\mathbf{x}_h |_{g'}))).$$

In order to turn the inductive step, Theorem \ref{onegap}, into a full expression for $\mathbf{x}_{f_S}$, Theorem \ref{formula}, we need to encode the dependence of $h'$ on $S$ and $g'$ into its notation. We do this by defining the \textit{elevation function} of $S$:

\begin{definition}
Given a multiset of vertices $S$, the \textbf{elevation function of $S$} is the function $h_S: S \to \mathbb{N}$ that maps $g \in S$ to
$$h_{S}(g) = \sum_{g' \in S \setminus \{g^1\}} \mathbbm{1}_{V(S_{g'})}(g).$$
If the same vertex appears multiple times in $S$, give them an arbitrary order so that they take consecutive values under $h_S$. For example, if the root appears three times in $S$, then they take the values $0, 1, 2$ under $h_S$, and a child of the root would take the value $3$.
\end{definition}

\begin{remark}
The elevation function $S$ can be thought of as sending $g \in V(T)$ to the number of elements of $S$ that are ancestors of $g$, possibly including itself.
\end{remark}

Since $h' = \sum_{\substack{g \in S}} \mathbbm{1}_{V(S_{g})}(g')$, we can further restate Theorem \ref{onegap} as follows:
$$\mathbf{x}_{f_{S'}} = \mathbf{x}_{f_S} \cdot \sigma( \tau( \sigma^{h'}(\mathbf{x}_h |_{g'}))) = \mathbf{x}_{f_S} \cdot \sigma( \tau( \sigma^{h_{S'}(g')}(\mathbf{x}_h |_{g'}))).$$

Finally, we have the theorem that expresses $\mathbf{x}_{f_S}$ in terms of $\mathbf{x}_h$ and $\mathbf{x}_h |_g$ for $g \in S$.

\begin{theorem} \label{formula}
For any multiset of vertices $S$, we have:
$$\mathbf{x}_{f_S} = \sigma \left( \mathbf{x}_h \tau \left( \prod_{g \in S} \sigma^{h_S(g)} (\mathbf{x}_h |_g) \right) \right).$$
\end{theorem}

\begin{proof}
For $S = \varnothing$, we have by definition that $\mathbf{x}_{f_\varnothing} = \sigma(\mathbf{x}_h)$, which is equal to the desired formula because $\tau(1) = 1$. Suppose that the desired formula is true for all $|S| = n$. Now, consider a multiset of vertices $S'$ such that $|S'| = n + 1$. Pick a $g' \in S'$ such that $g'$ has no descendants in $S'$, and let $S = S' \setminus \{{g'}^1\}$. By Theorem \ref{onegap}, we have the following. We use here the fact that $\sigma$ and $\tau$ are multiplicative.
\begin{align*}
\mathbf{x}_{f_{S'}} & = \mathbf{x}_{f_S} \cdot \sigma( \tau( \sigma^{h_{S'}(g')} (\mathbf{x}_h |_{g'}))) \\
& = \sigma \left( \mathbf{x}_h \tau \left( \prod_{g \in S} \sigma^{h_S(g)} (\mathbf{x}_h |_{g}) \right) \right) \cdot \sigma( \tau( \sigma^{h_{S'}(g')} (\mathbf{x}_h |_{g'}))) \\
& = \sigma \left( \mathbf{x}_h \tau \left( \prod_{g \in S'} \sigma^{h_{S'}(g)} (\mathbf{x}_h |_{g}) \right) \right).
\end{align*}

\end{proof}

\begin{example}
Consider the coloring together with the set $S$ in Figure \ref{fig:gaps}. In this example, no vertex of $S$ is an ancestor of another, so $h_S(g)$ is always zero. Then Theorem \ref{formula} says the following. It is instructive to split up $\tau$ to emphasize the shift that every individual gap produces.
\begin{align*}
\mathbf{x}_{f_S} & = \sigma \left( \mathbf{x}_h \prod_{g \in S} \tau(\mathbf{x}_h |_g) \right) \\
& = \sigma \left( {x_0}^1 {x_1}^4 {x_2}^6 {x_3}^8 \cdot \tau(x_1 x_2 x_3) \cdot \tau(x_2) \cdot \tau(x_2 x_3) \right) \\
& = \sigma \left( {x_0}^1 {x_1}^4 {x_2}^6 {x_3}^8 \cdot \frac{x_2 x_3 x_4}{x_1 x_2 x_3} \cdot \frac{x_3}{x_2} \cdot \frac{x_3 x_4}{x_2 x_3} \right) \\
& = \sigma \left( {x_0}^1 {x_1}^3 {x_2}^4 {x_3}^9 {x_4}^2 \right) \\
& = {x_1}^1 {x_2}^3 {x_3}^4 {x_4}^9 {x_5}^2,
\end{align*}
as expected. Note that the exponents of $\mathbf{x}_h |_g$ do not have to all be 1; they just happen to be so in this example.
\end{example}

\section{Reconstructing the tree} \label{main}

Our ultimate goal, Theorem \ref{fin}, is to reconstruct $T$ from the sampling function. Recall from Definition \ref{sampling} that the sampling function $F$ is defined by
$$F \left( \prod_{i \le n} x_i^{e_i} \right) = \max \left( \left[ \prod_{i \le n} x_i^{e_i} \right] \Gamma^<(T; \mathbf{x}) \right).$$

Our stepping stones to Theorem \ref{fin} involve reconstructing $\mathbf{x}_{(N)h}$.
This requires a recursive action: first reconstructing $\mathbf{x}_h$ in Theorem \ref{one}, then $\mathbf{x}_{\mathbf{x}_h}$ in Theorem \ref{two}, then $\mathbf{x}_{(3)h}$ in Theorem \ref{three}, and then general $\mathbf{x}_{(N)h}$ in Theorem \ref{general}.

The following theorem reconstructs $\mathbf{x}_h$ from $F(1) = {\max}^1(\Gamma^<(T; \mathbf{x}))$. Note that Section \ref{step1} is a specific case of this reconstruction.

The main idea of the proof is as follows.

We will show that the maximum $\mathbf{x}_f$ is achieved with $f_\varnothing$. See Figure \ref{fig:varnothing} for a depiction of $f_\varnothing$. We know that $f_\varnothing(v) = 1 + h_v$, so to find the number of vertices with coheight $n$, we need only check how many vertices are colored $1 + n$ in $f_\varnothing$.

\begin{theorem} \label{one}
The coheight profile $\mathbf{x}_h$ can be reconstructed from the sampling function $F$.
\end{theorem}

\begin{proof}

We evaluate $F$ at $1$. By definition, $F(1) = {\max}^1(\Gamma^<(T; \mathbf{x}))$. By Theorem \ref{f_S}, we need only find the $S$ that gives the maximum $\mathbf{x}_{f_S}$, so
$${\max}^1(\Gamma^<(T; \mathbf{x})) = \max_S \left( \mathbf{x}_{f_S} \right).$$
We apply Theorem \ref{formula} to express the latter in terms of $\mathbf{x}_h |_g$ terms:
$$\max_S \left( \mathbf{x}_{f_S} \right) = \max_S \left( \sigma \left( \mathbf{x}_h \tau \left( \prod_{g \in S} \sigma^{h_S(g)} (\mathbf{x}_h |_g) \right) \right) \right).$$
By Remark \ref{sigma_ordering}, $\sigma$ preserves the ordering of the elements of $\langle x_i \rangle_{i \in \mathbb{N}}$, so
$$\max_S \left( \sigma \left( \mathbf{x}_h \tau \left( \prod_{g \in S} \sigma^{h_S(g)} (\mathbf{x}_h |_g) \right) \right) \right) = \sigma \left( \mathbf{x}_h \max_S \left( \tau \left( \prod_{g \in S} \sigma^{h_S(g)} (\mathbf{x}_h |_g) \right) \right) \right),$$
and $\tau$ reverses said ordering:
$$\sigma \left( \mathbf{x}_h \max_S \left( \tau \left( \prod_{g \in S} \sigma^{h_S(g)} (\mathbf{x}_h |_g) \right) \right) \right) = \sigma \left( \mathbf{x}_h \tau \left( \min_S \left( \prod_{g \in S} \sigma^{h_S(g)} (\mathbf{x}_h |_g) \right) \right) \right).$$
The minimum is $1$, achieved by $S = \varnothing$, so we conclude with
\begin{align*}
\sigma \left( \mathbf{x}_h \tau \left( \min_S \left( \prod_{g \in S} \sigma^{h_S(g)} (\mathbf{x}_h |_g) \right) \right) \right) & = \sigma( \mathbf{x}_h \tau(1) ) = \sigma(\mathbf{x}_h).
\end{align*}

By Remark \ref{sigma_invertible}, $\sigma$ is invertible. Thus, we have reconstructed $\mathbf{x}_h$.
\end{proof}

To get more information out of the strict order quasisymmetric function, we need to exploit the sampling function more generally. Our method primarily involves repeated applications of the following operation.

To introduce Theorem \ref{conv}, we present a specific case of the theorem and its proof.

\begin{example}
We can reconstruct
\begin{equation}
\label{convonegap}
\min_S \left( \left[ x_n \right] \prod_{g \in S} \sigma^{h_S(g)} (\mathbf{x}_h |_g) \right),
\end{equation}
which equals $\mathbf{x}_h |_{g_1}$, where $g_1$ is the vertex of coheight $n$ with the smallest coheight profile.

We do this by isolating the following term from the sampling function $F$:
\begin{equation} \label{termonegap}
{\max}^1 \left( \left[ \sigma \left(\phi_n(\mathbf{x}_h) x_n^{-1} \right) \right] \Gamma^<(T; \mathbf{x}) \right).
\end{equation}

Let $f$ be the coloring such that (\ref{termonegap}) is $\mathbf{x}_f$. The condition means that $\mathbf{x}_f$ must match $\sigma(\mathbf{x}_h) = \mathbf{x}_{f_\varnothing}$ up to the exponent of $x_{n - 1}$, but the exponent of $x_n$ in $\mathbf{x}_f$ is 1 smaller. Thus, $f$ is identical to $f_\varnothing$ up to layer $n - 1$ but with one less use of color $n$, leaving a ``gap'' in layer $n$. With the maximality condition, $f$ after layer $n$ is also like $f_\varnothing$ except that the colors of the descendants of the gap are shifted up by one. Given this, it is possible to see that the maximality condition forces the gap to be $g_1$ (smallest coheight profile). See Figure \ref{fig:onegapagain} for a depiction of $f$. Then, it is possible to show that by comparing $\mathbf{x}_f$ with $\mathbf{x}_{f_\varnothing}$, we can reconstruct $\mathbf{x}_h |_{g_1}$. We omit the proofs here, as they will be detailed in Theorem \ref{conv}.

\begin{figure}[h]
\centering
\includegraphics{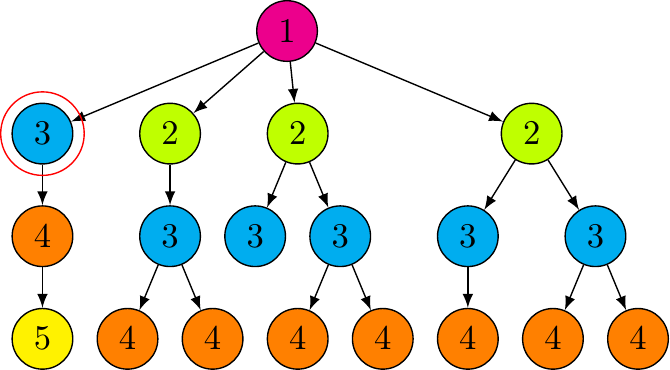}
\caption{Here the coloring $f$ as defined in Theorem \ref{two} is depicted. $g_1$ is circled.}
\label{fig:onegapagain}
\end{figure}

The idea behind the general case is to leave more gaps by imposing more restricted conditions in (\ref{termonegap}). For example, imposing the condition $[\sigma(\phi_n(\mathbf{x}_h) x_n^{-2})]$ in (\ref{termonegap}) leaves 2 gaps in layer $n$, equivalent to imposing the condition $[x_n^2]$ in (\ref{convonegap}).
\end{example}

\begin{theorem} \label{conv}
The function $\tilde{F} \colon \langle x_i \rangle_{i \in \mathbb{Z}^+} \to \langle x_i \rangle_{i \in \mathbb{Z}^+} \cup \{ \varnothing \}$ defined by
$$\tilde{F} \left( \prod_{i \le n} x_i^{e_i} \right) = \min_S \left( \left[ \prod_{i \le n} x_i^{e_i} \right] \prod_{g \in S} \sigma^{h_S(g)} (\mathbf{x}_h |_g) \right),$$
where the minimum is taken over all $S$ that produce a nonempty expression, can be reconstructed from the sampling function $F$.
\end{theorem}

Theorem \ref{conv} is the basis of the rest of this paper. With carefully chosen values of $\prod_{i \le n} x_i^{e_i}$, we can methodically reconstruct the information that we need to reconstruct $\mathbf{x}_{(N)h}$. For example, as will be seen in Theorem \ref{two}, we set $\prod_{i \le n} x_i^{e_i} = x_n$ in order to force $S$ to include a vertex of coheight $n$. Then, taking the minimum helps us get rid of everything extra, and we are left with $\mathbf{x}_h |_g$.

We need a quick definition for the proof.
\begin{definition}
For $n \in \mathbb{N}$, let the \textbf{truncate function} $\phi_n: \mathbb{Z}[[x_i]]_{i \in \mathbb{N}} \to \mathbb{Z}[[x_i]]_{i \in \mathbb{N}}$ be the function defined by
$$\phi_n \left( k \prod_{i \in \mathbb{N}} x_i^{e_i} \right) = k \prod_{i \le n} x_{i}^{e_i}.$$
\end{definition}

\begin{proof}[Proof of Theorem \ref{conv}]
Since we know $\mathbf{x}_h$ by Theorem \ref{one}, we can isolate the following term from the sampling function $F$.

\begin{equation*}
{\max}^1 \left( \left[ \sigma \left(\phi_n(\mathbf{x}_h) \prod_{i \le n} x_i^{e_{i - 1} - e_i} \right) \right] \Gamma^<(T; \mathbf{x}) \right).
\label{eq:reconstructable}
\end{equation*}

We will apply Theorem \ref{formula} to the expression and then modify the composition order of the functions, as shown. Like in Theorem \ref{one}, notice that $\sigma$ preserves the ordering of the elements of $\langle x_i \rangle_{i \in \mathbb{N}}$, while $\tau$ reverses it.
\begin{align*}
& {\max}^1 \left( \left[ \sigma \left(\phi_n(\mathbf{x}_h) \prod_{i \le n} x_i^{e_{i - 1} - e_i} \right) \right] \Gamma^<(T; \mathbf{x}) \right) \\
& = \max_S \left( \left[ \sigma \left( \phi_n(\mathbf{x}_h) \prod_{i \le n} x_i^{e_{i - 1} - e_i} \right) \right] \mathbf{x}_{f_S} \right) \\
& = \max_S \left( \left[ \sigma \left( \phi_n(\mathbf{x}_h) \prod_{i \le n} x_i^{e_{i - 1} - e_i} \right) \right] \sigma \left( \mathbf{x}_h \tau \left( \prod_{g \in S} \sigma^{h_S(g)} (\mathbf{x}_h |_g) \right) \right) \right) \\
& = \sigma \left( \max_S \left( \left[ \phi_n(\mathbf{x}_h) \prod_{i \le n} x_i^{e_{i - 1} - e_i} \right] \mathbf{x}_h \tau \left( \prod_{g \in S} \sigma^{h_S(g)} (\mathbf{x}_h |_g) \right) \right) \right) \\
& = \sigma \left( \mathbf{x}_h \max_S \left( \left[ \prod_{i \le n} x_i^{e_{i - 1} - e_i} \right] \tau \left( \prod_{g \in S} \sigma^{h_S(g)} (\mathbf{x}_h |_g) \right) \right) \right) \\
& = \sigma \left( \mathbf{x}_h \tau \left( \min_S \left( \left[ \prod_{i \le n} x_i^{e_i} \right] \prod_{g \in S} \sigma^{h_S(g)} (\mathbf{x}_h |_g) \right) \right) \right).
\end{align*}
In the last step, we use that $\tau(\prod_{i \le n} x_i^{e_i}) = \prod_{i \le n} x_i^{e_{i - 1} - e_i}$. Since $\sigma$ and $\tau$ are invertible by Remark \ref{sigma_invertible}, we reconstruct
$$\min_S \left( \left[ \prod_{i \le n} x_i^{e_i} \right] \prod_{g \in S} \sigma^{h_S(g)} (\mathbf{x}_h |_g) \right)$$
for any term $\prod_{i \le n} x_i^{e_i}$.
\end{proof}

\begin{remark} \label{truncate}
Note that one can freely add to or remove from $S$ any vertex $v$ satisfying $h_S(v) + h_v > n$, since this will not change whether the exponent of $x_i$ in the product is $e'_i$. Removing vertices from $S$ is guaranteed to decrease the product; thus, we know that the minimum $S$ has no removable vertices.
\end{remark}

We now reconstruct $\mathbf{x}_{\mathbf{x}_h}$ from the sampling function. Note that Section \ref{step2} is a specific case of the procedure below.

\begin{theorem} \label{two}
The coheight profile profile $\mathbf{x}_{\mathbf{x}_h}$ can be reconstructed from the sampling function $F$.

\end{theorem}

\begin{proof}
By Theorem \ref{conv}, we can reconstruct the following expression for all $n$ and $m$:
$$\tilde{F}(x_n^m) = \min_S \left( [x_n^m] \prod_{g \in S} \sigma^{h_S(g)} (\mathbf{x}_h |_g) \right).$$

The idea is to force $S$ to include no vertices with coheight $< n$ and $m$ vertices with coheight $n$. The $[x_n^m]$ condition achieves this. Then, by taking the minimum, we get rid of anything extra: we know from Remark \ref{truncate} that $S$ need not contain vertices with coheight $> n$, nor any repeats (else $h_S(g) \ne 0$). Then, the minimum will happen when $S = \{{g_i}^1 \mid 1 \le i \le m\}$, where $g_i$ is the vertex with coheight $n$ with the $i$th least coheight profile. Thus, we have

\[
\min_S \left( [x_n^m] \prod_{g \in S} \sigma^{h_S(g)} (\mathbf{x}_h |_g) \right) = \prod_{1 \le i \le m} \mathbf{x}_h |_{g_i}.
\]

Knowing this expression for every value of $m$, we can reconstruct $\mathbf{x}_h |_{g_m}$.

Splitting $\mathbf{x}_{\mathbf{x}_h}$ by coheight, we can write:
$$\mathbf{x}_{\mathbf{x}_h} = \prod_{v \in V(T)} x_{\mathbf{x}_h |_v} = \prod_n \prod_{\substack{v \in V(T) \\ h_v = n}} x_{\mathbf{x}_h |_v}$$
Then, for each $n$, the last product we can reconstruct from the $\mathbf{x}_h |_{g_m}$ that we have reconstructed.
\end{proof}

\begin{theorem} \label{three}
From the sampling function, we can reconstruct $\mathbf{x}_{(3)h}$.
\end{theorem}

\begin{proof} \let\qed\relax
For a certain coheight $n_0$, let $g_i$ be defined as in Theorem \ref{two}. By Theorem \ref{two}, we can reconstruct $\mathbf{x}_h |_{g_i}$ for every $i$. Thus, by Theorem \ref{conv}, we can reconstruct the following for all $m_0, n > n_0$, and $m$:
\begin{equation} \label{eq:threeterm}
\tilde{F} \left( \phi_{n} \left( \prod_{1 \le i \le m_0} \mathbf{x}_h |_{g_i} \right) x_{n}^{m} \right) = \min_S \left( \left[ \phi_{n} \left( \prod_{1 \le i \le m_0} \mathbf{x}_h |_{g_i} \right) x_{n}^{m} \right] \prod_{g \in S} \sigma^{h_S(g)} (\mathbf{x}_h |_g) \right).
\end{equation}

The idea here is to force $S$ to include $g_i \mid 1 \le i \le m_0$, and then in addition include $m$ vertices that satisfy $h_S(v) + h_v = n$. When taking the minimum, we run into the issue that the second requirement interferes with the first, which requires a combinatorial argument to resolve. Once the interference is resolved, we can use a technique similar to that used in the proof of Theorem \ref{two} to reconstruct the coheight profile of each vertex in $V(S_{g_i})$, and compiling all of these together gives us $\mathbf{x}_{(3)h}$.

Considering the $\left[ \phi_{n} \left( \prod_{1 \le i \le m_0} \mathbf{x}_h |_{g_i} \right) x_n^m \right]$ condition, the first nonzero exponent is $x_{n_0}^{m_0}$, which means that $S$ must include $m_0$ vertices with coheight $n_0$. Let $S_0$ be the set of these $m_0$ vertices. Then
\begin{equation} \label{eq:condition}
\phi_n \left( \prod_{g \in S_0} \mathbf{x}_h |_g \right) \le \phi_n \left( \prod_{g \in S} \sigma^{h_S(g)} (\mathbf{x}_h |_g) \right) = \phi_{n} \left( \prod_{1 \le i \le m_0} \mathbf{x}_h |_{g_i} \right) x_n^m,
\end{equation}
where the left inequality comes from $S_0 \subseteq S$ and the right equality comes from the condition. Recall that $\{{g_i}^1 \mid 1 \le i \le m_0\}$ is the $S_0$ with the least possible product of coheight profiles, or in other words
\begin{equation} \label{eq:min}
\min_{S_0} \left( \prod_{g \in S_0} \mathbf{x}_h |_g \right) = \prod_{1 \le i \le m_0} \mathbf{x}_h |_{g_i},
\end{equation}
which implies that
\begin{equation} \label{eq:min_ineq}
\phi_n \left( \prod_{g \in S_0} \mathbf{x}_h |_g \right) \ge \phi_n \left( \prod_{1 \le i \le m_0} \mathbf{x}_h |_{g_i} \right).
\end{equation}
Putting (\ref{eq:condition}) and (\ref{eq:min_ineq}) together, we have that
\begin{equation}
\label{eq:s0condition}
\phi_n \left( \prod_{1 \le i \le m_0} \mathbf{x}_h |_{g_i} \right) \le \phi_n \left( \prod_{g \in S_0} \mathbf{x}_h |_g \right) \le \phi_n \left( \prod_{1 \le i \le m_0} \mathbf{x}_h |_{g_i} \right) x_n^m.
\end{equation}
This forces $\phi_n \left( \prod_{g \in S_0} \mathbf{x}_h |_g \right)$ to be of the form $\phi_n \left( \prod_{1 \le i \le m_0} \mathbf{x}_h |_{g_i} \right) x_n^{m'}$ for some $0 \le m' \le m$.
Now, since by Theorem \ref{two} we know $\mathbf{x}_h |_{g_i}$ for all $i$, we know all the choices we have for $S_0$.

Let $S \setminus S_0 = S_1$. Consider the equality in (\ref{eq:condition}). We can split the left hand side into terms for $S_0$ and $S_1$ as follows:
\begin{equation*}
\phi_n \left( \prod_{1 \le i \le m_0} \mathbf{x}_h |_{g_i} \right) x_n^{m'} \cdot \phi_n \left( \prod_{g \in S_1} \sigma^{h_S(g)} (\mathbf{x}_h |_g) \right) = \phi_{n} \left( \prod_{1 \le i \le m_0} \mathbf{x}_h |_{g_i} \right) x_n^m,
\end{equation*}
which gives us
\begin{equation}
\label{eq:s1condition}
\phi_n \left( \prod_{g \in S_1} \sigma^{h_S(g)} (\mathbf{x}_h |_g) \right) = x_n^{m - m'}.
\end{equation}
As a consequence, the set $S_1$ contains no vertices with $h_S(v) + h_v < n$ and $m - m'$ vertices with $h_S(v) + h_v = n$. In addition, since we are taking a minimum in (\ref{eq:min}), we need not consider any vertices with $h_S(v) + h_v > n$ (Remark \ref{truncate}).

Now, we know that (\ref{eq:threeterm}) is equal to
\begin{equation} \label{eq:threetermconv}
\min_{S = S_0 \cup S_1} \left( \prod_{g \in S} \sigma^{h_S(g)}(\mathbf{x}_h |_g) \right) = \min_{S_0, S_1} \left( \left( \prod_{g \in S_0} \mathbf{x}_h |_g \right) \cdot \left( \prod_{g \in S_1} \sigma^{h_{S}(g)}(\mathbf{x}_h |_g) \right) \right),
\end{equation}
where the minimum is taken over $S_0$ satisfying (\ref{eq:s0condition}) and $S_1$ satisfying (\ref{eq:s1condition}).

One might hope that this choice of $S_0$ would be the same as the choice that produces a minimum value for $\prod_{g \in S_0} \mathbf{x}_h |_g$. If this were the case, then we'd be guaranteed $S_0 = \{{g_i}^1 \mid 1 \le i \le m_0\}$ by Theorem \ref{two}. However, this is not true. Notice that the possibilities for $S_1$ depend on $S_0$ due to the $h_S(v)$ term. Thus, it might be the case that the minimal set $S_1$ of a non-minimal set $S_0$ produces a smaller value than the minimal set $S_1$ of the minimal set $S_0$.
We call this non-minimality issue the ``swapping problem.''

It is worth noting that the $h_S(v)$ term is necessary, because otherwise it would be impossible to determine whether a vertex $v$ is a descendant of a vertex of $S_0$. This would give no additional information past Theorem \ref{two}.

We can still determine the information that we want, which is the coheight profile of each vertex in $V(S_{g_i})$. We pause the proof here to set up another framework in Section \ref{framework2}, which we will use to describe the combinatorial procedure that cleans up our information. We defer the rest of the proof to Section \ref{main2}.

\end{proof}

\section{A framework to resolve Theorem \ref{three}} \label{framework2}

With this framework, we aim to describe the combinatorial procedure that cleans up the information reconstructed in the proof of Theorem \ref{three}. We resolve the swapping problem with a strategy that we term ``predict and verify.''

Recall that the information is as follows: for every choice of $n_0, m_0, n, m$, we know expression (\ref{eq:threetermconv}). Throughout this section, we fix $n_0$. We wish to determine, for each vertex $g_i$ with coheight $n_0$, the coheight profile $\mathbf{x}_h |_v$ of each $v \in V(S_{g_i})$.

Let us first set up some notation.

\begin{itemize}
    \item Let $L_k$ be the set of vertices with coheight $k$.

    \item Let $L_k(g)$ be the set of descendants of $g$ with coheight $k$.
    \item Let $L_k(S_0)$ be the set of descendants of vertices in $S_0$ with coheight $k$.

    \item Let $v_i$ be the element of $L_n$ with the $i$th least coheight profile.
\end{itemize}

We also set up the following definitions.

\begin{definition}
A \textbf{candidate for $S_0$} is a vertex that satisfies condition (\ref{eq:s0condition}). These are the vertices that could possibly be in $S_0$. We denote the set of candidates for $S_0$ by $C_0$.
\end{definition}

\begin{definition}
For a given set $S_0$, a \textbf{candidate for $S_1$} is a vertex that satisfies condition (\ref{eq:s1condition}). These are the vertices that could possibly be in $S_1$. We denote the set of candidates for $S_0$ by $C_1(S_0)$.
\end{definition}

For a given $S_0$, the set $C_1(S_0)$ consists of vertices that satisfy $h_{S_0}(v) + h_v = n$. These include $L_{n - 1}(S_0)$ as well as $L_n \setminus L_n(S_0)$. We invoke induction on $n$ so that our inductive hypothesis is the following: for any vertex $g \in C_0$, we know the coheight profiles of the set $L_{n - 1}(g)$. By Theorem \ref{two}, we also know the coheight profiles of $L_n$. Thus, our task is to determine for any $g \in C_0$ which of the vertices in $L_n$ are in $L_n(g)$. As we are inducting on $n$, we fix $n$ from now on.

This goal can be more easily discussed with the following definition.

\begin{definition}
For each vertex $v_i$, the \textbf{position} of $v_i$ is the vertex $g \in C_0$ such that $v_i \in L_n(g)$.
\end{definition}

With this definition, our goal is to determine the position of each vertex $v_i \in L_n$.

We determine the positions of $v_1, \dots, v_{|L_n|}$ inductively. To determine the position of $v_i$, we make certain choices of $m_0$ and $m$ such that we can predict (\ref{eq:threetermconv}) with our current knowledge, assuming $v_i \not \in L_n(S_0)$ for the predicted $S_0$. Then, we show that $v_i \not \in L_n(S_0)$ if and only if the prediction is correct. In addition, we show that knowing whether $v_i \in L_n(S_0)$ for the sets $S_0$ that are involved in these predictions is sufficient to narrow down $v_i$ to one possible position.

We set up a few definitions to enable us to work with (\ref{eq:threetermconv}) more easily. Note that these definitions are not wholly rigorous; they are meant to be a guide and will be modified throughout the section.

Expression (\ref{eq:threetermconv}) takes the minimum of the product of two terms. The first term is encapsulated in the following definition.

\begin{definition}
The \textbf{padding} on $S_0$, denoted $\mathcal{P}(S_0)$, is $\prod_{g \in S_0} \mathbf{x}_h |_{g}$.
\end{definition}

The second term is encapsulated in the following definition.

\begin{definition}
Fix an $S_0$. Recall that $C_1(S_0) = L_{n - 1}(S_0) \cup L_n \setminus L_n(S_0)$. The \textbf{stack} on $S_0$, denoted $\mathcal{S}(S_0)$, is the sequence defined by the set
$$\{ \sigma^{h_{S_0}(g)}(\mathbf{x}_h |_v) \mid v \in L_{n - 1}(S_0) \cup L_n \setminus L_n(S_0) \}$$
arranged from least to greatest. In the case of equalities, we place elements of $L_{n - 1}(S_0)$ before elements of $L_n \setminus L_n(S_0)$.

We let the $i$th element of $\mathcal{S}(S_0)$ be $\mathcal{S}_i(S_0)$.
\end{definition}

Importantly, notice that we do not currently know all of $\mathcal{S}(S_0)$. We know only the elements that are less than $\mathbf{x}_h |_{v_1}$, because it is uncertain whether $v_1 \in L_n \setminus L_n(S_0)$.

\begin{definition}
The \textbf{$k$th partial product} of the stack on $S_0$, denoted $\prod_{i = 1}^k \mathcal{S}_i(S_0)$, is the product of the first $k$ elements of the stack.
\end{definition}

Using the above definitions, we rewrite (\ref{eq:threetermconv}) as
\begin{equation} \label{eq:news1condition}
\min_{S = S_0 \cup S_1} \left( \prod_{g \in S} \sigma^{h_S(g)}(\mathbf{x}_h |_g) \right) = \min_{S_0, S_1} \left( \mathcal{P}(S_0) \cdot \prod_{i = 1}^{|S_1|} \mathcal{S}_i(S_0) \right).
\end{equation}
We only know the value of this if all the terms $\mathcal{S}_i(S_0)$ for $1 \leq i \leq |S_1|$ are less than $\mathbf{x}_h |_{v_1}$. If any of the terms is bigger than $\mathbf{x}_h |_{v_1}$, then we must consider whether $\mathbf{x}_h |_{v_1}$ is in each stack. Our strategy is to first compare partial products until we have to consider $v_1$, and then determine the stacks containing $v_1$.

We can now define a prediction using the above definitions.

\begin{definition}
Though it cannot be true, assume that $\mathbf{x}_h |_{v_1}$ is in every stack $\mathcal{S}(S_0)$. Then, for a choice of $m_0$ and $m$, we let the \textbf{predicted $S_0$} and \textbf{predicted $S_1$} be the sets $S_0$ and $S_1$ that produce the minimum in (\ref{eq:news1condition}) under the above assumption. We say that the prediction is \textbf{correct} if the value of (\ref{eq:news1condition}) under the above assumption is equal to the actual value of (\ref{eq:news1condition}).
\end{definition}

In the following theorem, we show that it is possible to make a certain prediction whose correctness determines whether $v_1 \in L_n(g)$ for some $g \in C_0$.

\begin{theorem} \label{existsnice}
For $m_0 = 1$, there exists an integer $m$ such that the largest vertex in the predicted $S_1$ is $v_1$.
\end{theorem}

\begin{proof}
Though it cannot be true, assume throughout this proof that $\mathbf{x}_h |_{v_1}$ is in every stack $\mathcal{S}(g)$.

The predicted $S_0$ must consist of a single vertex $g \in C_0$.

For each $m$, this $g$ must be the $g \in C_0$ that gives the smallest $\mathcal{P}(g) \cdot \prod_{i = 1}^{m - m'} \mathcal{S}_i(g)$. Note that $m'$ is a function of $g$, so our partial products are not lined up. In order to line them up, we lift each stack $\mathcal{S}(g)$ up by $m'$ elements; that is, we increase the index of each element in $\mathcal{S}(g)$ by $m'$. We leave the $k$th partial product for each $k < m'$ undefined.

For each $g \in C_0$, we want to consider $m$ for which $\mathcal{S}_m(g) = \mathbf{x}_h |_{v_1}$;

specifically the largest such $m$, since we want to pick out $\mathbf{x}_h |_{v_1}$ from other equal elements (remember that in the case of equalities, we place elements of $L_{n - 1}(S_0)$ before elements of $L_n \setminus L_n(S_0)$). Thus we make the following definitions:

\begin{definition}
For a positive integer $m$, the \textbf{minimal gap at $m$}, denoted $g_m$, is the $g \in C_0$ that minimizes the number $\mathcal{P}(g) \cdot \prod_{i = 1}^{m} \mathcal{S}_i(g)$.
\end{definition}

\begin{definition}
The \textbf{critical index} of $g$, denoted $m_g$, is the largest $m$ for which $\mathcal{S}_m(g) = \mathbf{x}_h |_{v_1}$.
\end{definition}

To prove Theorem \ref{existsnice}, we want to show that for some $m$, the predicted $S_1$ has largest vertex $v_1$. In terms of the above definitions, this $m$ needs to satisfy two criteria: it's the critical index of some $g \in C_0$ (so that the largest vertex of the predicted $S_1$ is $v_1$), and this $g$ is the minimal gap at $m$ (so that this $g$ is actually the predicted). Thus, we want to show that some $m$ satisfies $m_{g_m} = m$.

We begin from $m = 1$ and increment upward. At every step, we have three possibilities:

\begin{enumerate}
    \item $m_{g_m} < m$
    \item $m_{g_m} = m$
    \item $m_{g_m} > m$
\end{enumerate}

If 2) $m_{g_m} = m$ is true, then we are done, and we stop the procedure. Thus, the procedure only continues if 1) $m_{g_m} < m$ or 3) $m_{g_m} > m$ is true. We claim that 1) $m_{g_m} < m$ is never true. Since 3) $m_{g_m} > m$ cannot be true for the maximal critical index, the procedure must eventually stop.

To show that 1) $m_{g_m} < m$ is never true, we proceed by induction. The statement 1) is trivially false for the first critical index. For the inductive step, suppose that 1) $m_{g_m} < m$ is true. Let us now consider $m - 1$.

By the inductive hypothesis, we must have 3) $m_{g_{m - 1}} > m - 1$.

By the definition of $g_{m - 1}$, we know that
\begin{equation} \label{eq:inductive_hyp}
\mathcal{P}(g_m) \cdot \prod_{i = 1}^{m - 1} \mathcal{S}_i(g_m) \ge \mathcal{P}(g_{m - 1}) \cdot \prod_{i = 1}^{m - 1} \mathcal{S}_i(g_{m - 1}).
\end{equation}
Now, let us look at the definition of critical index.
\begin{itemize}
    \item Since $m_{g_m}$ is the critical index of $g_m$, we know that $\mathcal{S}_i(g_m) \ge \mathbf{x}_h |_{v_1}$ for $i \ge m_{g_m}$. We assumed above that $m_{g_m} < m$, so we know that $i \ge m_{g_m}$ is sufficient for $i \ge m$.
    \item Since $m_{g_{m - 1}}$ is the critical index of $g_{m - 1}$, we know that $\mathcal{S}_i(g_{m - 1}) \le \mathbf{x}_h |_{v_1}$ for $i \le m_{g_{m - 1}}$. We determined above that $m_{g_{m - 1}} > m - 1$, so we know that $i \le m_{g_{m - 1}}$ is sufficient for $i \le m$.
\end{itemize}
The two conditions overlap at $i = m$. Thus, we know that $\mathcal{S}_m(g_m) \ge \mathbf{x}_h |_{v_1} \ge \mathcal{S}_m(g_{m - 1})$. Multiplying (\ref{eq:inductive_hyp}) with the above, we get
$$\mathcal{P}(g_m) \cdot \prod_{i = 1}^{m} \mathcal{S}_i(g_m) \ge \mathcal{P}(g_{m - 1}) \cdot \prod_{i = 1}^{m} \mathcal{S}_i(g_{m - 1}).$$
This contradicts the definition of $g_m$. Thus, 1) could not have been true. This completes the proof of Theorem \ref{existsnice}.

Notice that in making our prediction, we did not need to know any elements $\mathcal{S}_i(S_0)$ larger than $\mathbf{x}_h |_{v_1}$; we just needed to know that they were larger.
\end{proof}

The following definition will allow us to discuss the value of $m$ determined in Theorem \ref{existsnice}.

\begin{definition}
Theorem \ref{existsnice} states the existence of an integer $m$ such that the largest vertex in the predicted $S_1$ is $v_1$. Let this $m$ be the \textbf{nice $m$}. For $m_0 = 1$ and the nice $m$, the predicted $S_0$ contains one element $g$. Let this $g$ be the \textbf{nice $g$}.
\end{definition}

\begin{remark} \label{newordering}
Theorem \ref{existsnice} produces an ordering on the $g \in C_0$ as follows. We let the smallest $g$ be the nice $g$. Then, remove this $g$. We can apply Theorem \ref{existsnice} again to get another nice $g$. Let this be the second smallest $g$. We proceed in this way until all the $g \in C_0$ are used up, and this produces an ordering on the vertices $g$. This will be important later.
\end{remark}

Now, we want to show that for $m_0 = 1$ and the nice $m$, the correctness of the prediction determines whether $v_1 \in L_n(g)$ for some $g \in C_0$.

\begin{theorem} \label{prediction}
For $m_0 = 1$ and the nice $m$, the prediction is correct if and only if $v_1$ is not in $L_n(g)$, where $g$ is the nice $g$.
\end{theorem}

\begin{proof}
Let $g_0$ be the position of $v_1$.

If the nice $g$ is $g_0$, then the prediction will be incorrect because $\mathbf{x}_h |_{v_1}$ is not actually in $\mathcal{S}(g)$, as we had assumed.

If the nice $g$ is not $g_0$, then the only change that comes about from assuming $\mathbf{x}_h |_{v_1}$ is in every stack is that the $i$th partial product of $\mathcal{S}(g_0)$ increases for $i$ where $\mathcal{S}_i(g_0) \ge \mathbf{x}_h |_{v_1}$. This would not change any of the predictions for $g_m$: for $m < m_{g_0}$, nothing changes, and for $m > m_{g_0}$, we know that $g_m \ne g_0$ (otherwise $m_{g_m} < m$, contradicting the claim within Theorem \ref{existsnice}) and so increasing a partial product of $\mathcal{S}(g_0)$ would not change predictions for $g_m$. Thus, the prediction would be correct.
\end{proof}

We generalize to $m_0 \ne 1$.

\begin{theorem} \label{existsnice_n}
For fixed $m_0 \ne 1$, there exists an $m$ such that the predicted $S_1$ has largest vertex $v_1$, and in addition, the predicted $S_0$ consists of the $m_0$ least $g \in C_0$.
\end{theorem}

\begin{proof}
Though it cannot be true, assume throughout this proof that $\mathbf{x}_h |_{v_1}$ is in every stack $\mathcal{S}(S_0)$.

Let us begin with new definitions of minimal gap (set) and critical index:

\begin{definition}
For a positive integer $m$, the \textbf{minimal gap set} at $m$ is the set $S_0$ of $m_0$ elements $g \in C_0$ that minimizes $\mathcal{P}(S_0) \cdot \prod_{i = 1}^{m} \mathcal{S}_i(S_0)$.
\end{definition}
\begin{definition}
The \textbf{critical index} of $S_0$, denoted $m_{S_0}$, is the largest $m$ for which $\mathcal{S}_m(S_0) = \mathbf{x}_h |_{v_1}$.
\end{definition}

In addition, let $M$ be the set containing the $m_0$ least elements $g \in C_0$. We claim that $M = g_{m_M}$.

Consider any other set $S_0 \ne M$ of $m_0$ elements $g \in C_0$. We pair $g \in M$ with $g' \in S_0$ if they are both the $i$th least element in their respective sets. By the definition of $M$, we know that the pairs $(g, g')$ satisfy
$$\mathcal{P}(g) \cdot \prod_{i = 1}^{m_g} \mathcal{S}_i(g) \le \mathcal{P}(g') \cdot \prod_{i = 1}^{m_g} \mathcal{S}_i(g').$$

Multiplying the equations together for all $g$, we get
$$\mathcal{P}(M) \cdot \prod_{(g, g')} \prod_{i = 1}^{m_g} \mathcal{S}_i(g) \le \mathcal{P}(S_0) \cdot \prod_{(g, g')} \prod_{i = 1}^{m_g} \mathcal{S}_i(g').$$

Since the left hand side of the above inequality contains exactly the terms of $\mathcal{S}_i(g)$ ($g \in M$) that are less than or equal to $\mathbf{x}_h |_{v_1}$, it is equal to $\mathcal{P}(M) \cdot \prod_{i = 1}^{m_M} \mathcal{S}_i(M) \cdot (\mathbf{x}_h |_{v_1})^{m_0 - 1}$. For the right hand side, if we only consider elements $\le \mathbf{x}_h |_{v_1}$, notice that $\mathcal{S}(S_0) = \bigcup_{g \in S_0} \mathcal{S}(g)$. Thus, we must have that $\prod_{(g, g')} \prod_{i = 1}^{m_g} \mathcal{S}_i(g') \le \prod_{i = 1}^{m_M} \mathcal{S}_i(S_0) \cdot (\mathbf{x}_h |_{v_1})^{m_0 - 1}$.
\end{proof}

\begin{theorem}
Theorems \ref{existsnice}-\ref{existsnice_n} are also true if $v_1$ is replaced with $v_i$, for any $i$.
\end{theorem}

\begin{proof}
Generalized Theorem \ref{existsnice} and \ref{prediction} work trivially for general $v_i$: since we already know the position of $v_j$ for $j < i$, we know every stack $\mathcal{S}(g)$ up until $\mathbf{x}_h |_{v_i}$, which is enough to predict whether the position of $v_i$ is $g$.

For generalized Theorem \ref{existsnice_n}, we apply a transformation to the stacks and then proceed in a similar fashion to Theorem \ref{existsnice_n}. The transformation is as follows:

Delete $\mathbf{x}_h |_{v_j}$ from each stack. As the stacks that we consider are up to $\mathbf{x}_h |_{v_i}$, every $\mathbf{x}_h |_{v_j}$ for $j < i$ is guaranteed to be included, so the deletion operation preserves the inequalities in Theorem \ref{existsnice_n}. For the unique stack $\mathcal{S}(g_0)$ that we have already determined does not contain $\mathbf{x}_h |_{v_j}$, we divide the padding $\mathcal{P}(g_0)$ by $\mathbf{x}_h |_{v_j}$. This preserves the rule that $\mathbf{x}_h |_{v_j} \in \mathcal{S}(S_0)$ if and only if $g_0 \not \in S_0$.
\end{proof}

\begin{theorem}
The sets $S_0$ we predict are sufficient to narrow down $v_i$ to one position.
\end{theorem}

\begin{proof}
By generalized Theorem \ref{prediction}, the position of $v_1$ is one of the $m_0$ least elements of $C_0$ if and only if the prediction for $m_0$ is incorrect.

Thus, we can determine the position of $v_1$ as follows. We check our predictions for $m_0 = 1, 2, \dots$ until we get one that is incorrect: let this be $m_f$. Then the position of $v_1$ must be one of the $m_f$ least elements of $C_0$, and it cannot be any of the $m_f - 1$ least elements of $C_0$. Thus, it must be precisely the $m_f$th least element of $C_0$.
\end{proof}

\section{Reconstructing the tree, continued}
\label{main2}

\begin{proof}[Proof of Theorem \ref{three}, continued]
Now, we know the coheight profiles of every vertex in $V(S_{g_i})$ for each $g_i$. Thus, we know the coheight profile profile $\mathbf{x}_{\mathbf{x}_h} |_{g_i}$ of each $g_i$.

We can proceed to find $\mathbf{x}_{(3)h}$ for the entire tree via the definition:
$$\mathbf{x}_{(3)h} = \prod_{v \in V(T)} \mathbf{x}_{\mathbf{x}_h} |_v.$$
\end{proof}

\begin{theorem} \label{general}
From the sampling function, we can reconstruct $\mathbf{x}_{(N)h}$ for any positive integer $x$.
\end{theorem}

\begin{proof}
We can recursively perform something analogous to Theorem \ref{three} in order to reconstruct general $\mathbf{x}_{(N)h}$. Rather than just having $S_0$ and $S_1$, we also have $S_2$, $S_3$, up to $S_{x - 2}$. The expression we consider is
\begin{equation*}
\min_S \left( \left[ \phi_{n} \left( \prod_{0 \le x' \le x - 3} \prod_{g \in S_{x'}} \sigma^{h_S}(\mathbf{x}_h |_g) \right) x_{n}^{m} \right] \prod_{g \in S} \sigma^{h_S(g)} (\mathbf{x}_h |_g) \right).
\end{equation*}
Suppose that we have already determined $\mathbf{x}_{(x - 1)h}$, so we know the possibilities for $S_0, \dots, S_{x - 3}$. Suppose also that we are working inductively, so that we already know the positions of some of the candidates for $S_{x - 2}$.

Out of the candidates for $S_{x - 2}$ with undetermined position, let $v_i$ be the one with $i$th smallest coheight profile. Using an argument similar to Theorems \ref{existsnice} through \ref{existsnice_n}, we have that for fixed $S_0, \dots, S_{x - 4}$, we can find an $S_{x - 3}$ with $|S_{x - 3}| = 1$ and $|S_{x - 2}|$ such that the largest element in the predicted $S_{x - 2}$ is $v_1$.

Via Remark \ref{newordering}, we can extend this to an ordering of the candidates for $S_{x - 3}$. Now, out of the candidates for $S_{x - 3}$ with undetermined position, let $V_i$ be the $i$th smallest vertex under the above ordering.

Then, we allow $S_{x - 4}$ to vary. Using an argument similar to Theorems \ref{existsnice} through \ref{existsnice_n}, we have that for fixed $S_0, \dots, S_{x - 5}$, we can find an $S_{x - 4}$ with $|S_{x - 4}| = 1$ and $|S_{x - 3}|$ such that the predicted $S_{x - 3}$ has largest element $V_1$. (Note: the essential reason why this argument works is that the ordering of Remark \ref{newordering} has the additive property described in Theorem \ref{existsnice_n}.)

Via Remark \ref{newordering}, we can extend this to an ordering of the candidates for $S_{x - 4}$. Then, we allow $S_{x - 5}$ to vary, and so on.

The final collection $S_0, \dots, S_{x - 3}, |S_{x - 2}|$ we find is the one we try first, and whether our prediction is correct or not tells us whether $v_1 \in L_n(S_{x - 3})$. By the inductive hypothesis, we already knew whether $v_1 \in L_n(S_{x - 3} \setminus \{{V_1}^1\})$, so we now know whether $v_1 \in L_n(V_1)$.

Next, we do the same procedure for $V_2$, and we can determine whether $v_1 \in L_n(V_2)$. We continue like this to determine the location of $v_1$. This directly generalizes to general $v_i$.
\end{proof}

Finally, we prove our main theorem.

\begin{proof}[Proof of Theorem \ref{fin}]
By Theorem \ref{general}, we can reconstruct $\mathbf{x}_{(N)h}$ for any positive integer $N$. This will be enough to reconstruct $T$.

We invoke recursion on the number of layers in $T$.

\begin{itemize}
    \item If $T$ has 2 layers, we can reconstruct $T$ from $\mathbf{x}_{h}$, since it suffices to know the number of children of the root.
    \item Suppose that for some $n \ge 2$, the following is true: if $T$ has $n$ layers, then we can reconstruct $T$ from $\mathbf{x}_{(n - 1)h}$. We claim that if $T$ has $n + 1$ layers, then we can reconstruct $T$ from $\mathbf{x}_{(n)h}$. Note that the subtree induced by a child of the root has at most $n$ layers. Since knowing $\mathbf{x}_{(n)h}$ gives us $\mathbf{x}_{(n - 1)h}$ of each child of the root, we can reconstruct each child's induced subtree, whose connection with the root completes the reconstruction of $T$. \qedhere
\end{itemize}
\end{proof}

\section{Future directions} \label{future}

Hasebe and Tsujie actually proved that the strict order quasisymmetric function distinguishes not only rooted trees, but also $(\mathsf{N}, \bowtie)$-free posets, which is a class of posets that includes but is not limited to rooted trees \cite{hasebe2017order}. We are interested to see if an analogue of our formalization and/or procedure exists in this broader setting.

Awan and Bernardi define the quasisymmetric $B$-polynomial, which is a simultaneous generalization of the chromatic quasisymmetric function and the Tutte symmetric function \cite{awan2020tutte}. They pose a number of open questions about the invariant. Question 10.6 part (ii) is resolved by our result or equally by the result of Hasebe and Tsujie. We are curious whether parts (iii) and (iv) of the question are resolvable using a similar sampling method as this paper. In general, because our combinatorial approach significantly differs from the algebraic approaches of many other papers in algebraic combinatorics, there may be results that are only within reach via our method.

Another direction to explore is looking for situations similar to the swapping problem in Section \ref{framework2} and then applying the ``predict and verify'' strategy. The swapping problem can be stated in a more general context as the following:

Suppose we have a totally ordered abelian group $R$, and $A_i, B_i$ are multisets with elements from $R$. Let $a_i(n)$ be the sum of the $n$ least elements of $A_i \cup \bigcup_{j \ne i} B_j$, and let $\delta_i \in R$ be constants. Given $A_i$, $S = \bigcup B_i$ and $s_n = \min^1(a_i(n) + \delta_i)$ for $1 \le n \le \max^1(|A_i| + |S| - |B_i|)$, can we determine each individual $A_i$?

Our resolution to the problem applies in this general situation as well. Thus, any problem that reduces to this general situation can be solved with our method.

\section{Acknowledgements}

I would like to thank Prof.\ Andrew Blumberg, Yongyi Chen, Prof.\ Pavlo Pylyavskyy, and Dr.\ Shuhei Tsujie for their invaluable advice. Last but not least, I would like to thank Prof.\ Pavel Etingof, Dr.\ Slava Gerovitch, and Dr.\ Tanya Khovanova of MIT PRIMES for providing me with the research opportunity that inspired this project.

\printbibliography[heading=bibnumbered]

\end{document}